\documentclass[reqno,11pt]{amsart}
\usepackage[numbers,square,comma,sort&compress]{natbib}
\usepackage{amsmath,amssymb}
\usepackage{color}
\usepackage{cases}
\usepackage{graphicx}

\newcommand{\R}{\mathbb{R}}

\newcommand{\diff}{\,\mathrm{d}}

\let\le=\leqslant
\let\ge= \geqslant

\newcommand\txt{\textstyle}

\newtheorem{theorem}{Theorem}

\newtheorem{corollary}[theorem]{Corollary}
\newtheorem{remark}[theorem]{Remark}

\usepackage[margin={3cm,3cm}]{geometry}

\begin{document}

\title[Dynamic pressure in a solitary wave]
{Extrema of the dynamic pressure in a solitary wave}

\author[F. Genoud]{Fran\c cois Genoud}
\address{Delft Institute of Applied Mathematics \\
Delft University of Technology \\
Mekelweg 4 \\
2628~CD Delft \\ The Netherlands}
\email{s.f.genoud@tudelft.nl}

\begin{abstract}
We study the dynamic pressure in an irrotational solitary wave propagating at the surface of
water over a flat bed, under the influence of gravity. We consider the nonlinear regime, that is,
the case of waves of moderate to large amplitude. We prove that, independently of the wave amplitude,
the maximum of the dynamic pressure is attained at the wave crest, while its minimum is attained
at infinity.
\end{abstract}

\maketitle


\section{Introduction}

Solitary waves can propagate on a free surface of water over a flat bed under the influence of gravity
over long distances, while maintaining a constant shape. They are two-dimensional objects
in the sense that they present essentially no variations in the horizontal direction perpendicular
to the direction of propagation of the wave. They are thus fully characterised by a vertical cross section
parallel to the direction of propagation, where they take the form of a single hump of elevation of the water,
moving at constant speed. In the reference frame moving with the wave, their profile is steady and 
symmetric with respect to the vertical axis through the wave crest, and decreases rapidly away from this
axis, see Fig.~\ref{fig1}. 
The first observation of this phenomenon was made by Scott Russell in 1834 in a canal near
Edinburgh. His report and subsequent laboratory experiments played an important role in the early
developments of water waves theory, see \cite{craik}. 

We are interested here in the description of solitary waves of moderate and large amplitude, which fail
to be captured accurately by the linear theory of water waves. As the effects of surface tension and 
viscosity are negligible in this regime, and it is physically reasonable to assume the water has a constant
density $\rho$ (we take $\rho=1$ throughout), the waves are described by
the Euler equation for a homogeneous incompressible fluid with free boundary, 
under the influence of gravity only. Moreover, motivated by the observation that 
large areas of abyssal plains in the oceans are essentially flat, 
we will restrict our attention here to a body of water over a flat bed.

The rigorous mathematical investigation, beyond the linear theory approximation, 
of the fluid motion beneath gravity water waves has made important progress in recent years,
see \cite{const_esch,const_invent,const_strauss,const_esch_hsu} and references in these papers.
The study of the pressure in the fluid is of particular interest, from a theoretical point of view but also
due to its important practical applications in maritime engineering. A good knowledge of the pressure
field is indeed essential to compute the forces acting on maritime structures. On the other hand,
pressure measurements at the bottom of the water can also be used to infer precious information about the 
waves on the surface \cite{const_estimate, const_recovery, oliveras, florian, hsu}.

In the context of irrotational waves, the pressure field was investigated 
by Constantin and Strauss \cite{const_strauss} for periodic waves,
and Constantin, Escher and Hsu \cite{const_esch_hsu} for solitary waves. Their main results
concern the monotonicity properties of the pressure in the fluid domain. Although
fluid motion can be driven by pressure gradients, in water at rest the {\em hydrostatic pressure}
-- which has a constant vertical gradient pointing downwards -- only counterbalances gravity
and does not induce any motion. The study of the relation between fluid motion and pressure 
therefore benefits from introducing the {\em dynamic pressure}, which is 
defined as the difference between the total pressure in the fluid and the hydrostatic pressure,
see \eqref{dynamic}. The dynamic pressure beneath irrotational periodic gravity
water waves was recently investigated by Constantin in \cite{const}, where it is proved 
that the maximum of the dynamic pressure
is attained at the wave crest and its minimum at the wave trough. Since it is known \cite{amick_toland}
that periodic waves converge to solitary waves in the long-wave limit, it is natural to guess 
that similar results hold for solitary waves. (Note that one should be careful when applying this
kind of reasoning to dynamic properties of the waves, for it was shown in \cite{const_esch} that
particle trajectories in the fluid undergo a dramatic qualitative change in the long-wave limit.) 
In the present paper we prove, using  
maximum principles for elliptic partial differential equations, that it is indeed the case. Namely,
the maximum of the dynamic pressure in an irrotational solitary wave is attained at the crest, while
its minimum is attained at infinity.

\subsection*{Acknowledgement} I am grateful to Adrian Constantin for drawing my attention
to this problem, and to the anonymous referees who helped improve the 
presentation of the paper.


\section{Mathematical formulation of the problem}\label{model.sec}

Irrotational solitary gravity water waves are two-dimensional. 
It was indeed proved in \cite{craig2} that, in the absence of vorticity,
no truly three-dimensional solitary waves can exist.
A travelling solitary wave is thus fully characterised by the description of 
a cross section of the flow, perpendicular to the crest line. 
We choose Cartesian 
coordinates $(X,Y)$, the $Y$-axis pointing vertically upwards, the $X$-axis being parallel to the 
direction of propagation of the wave. We require the flow to be at rest for $X\to\pm\infty$, and we
choose the $Y$ coordinate so that $Y=0$ there, with the flat bed lying at depth $Y=-d, \ d>0$.
We suppose that the crest of the wave is at $X=0$ at time $t=0$.

We investigate the dynamic pressure in a
permanent wave with profile $Y=\eta(X-ct)$, moving at constant speed $c>0$, 
so we assume that the velocity field of the flow has the form
\[
(u,v)=(u(X-ct),v(X-ct)).
\]
Under these assumptions, time can be removed from the governing equations by describing
the wave in the moving frame, that is, in the coordinates
\[
x=X-ct, \quad y=Y.
\]
In the new reference frame, which moves at speed $c$ in the direction of propagation of the 
wave, the wave is stationary and the flow is steady, see Fig.~\ref{fig1}.

\begin{figure}[th]
\includegraphics[width=130mm]{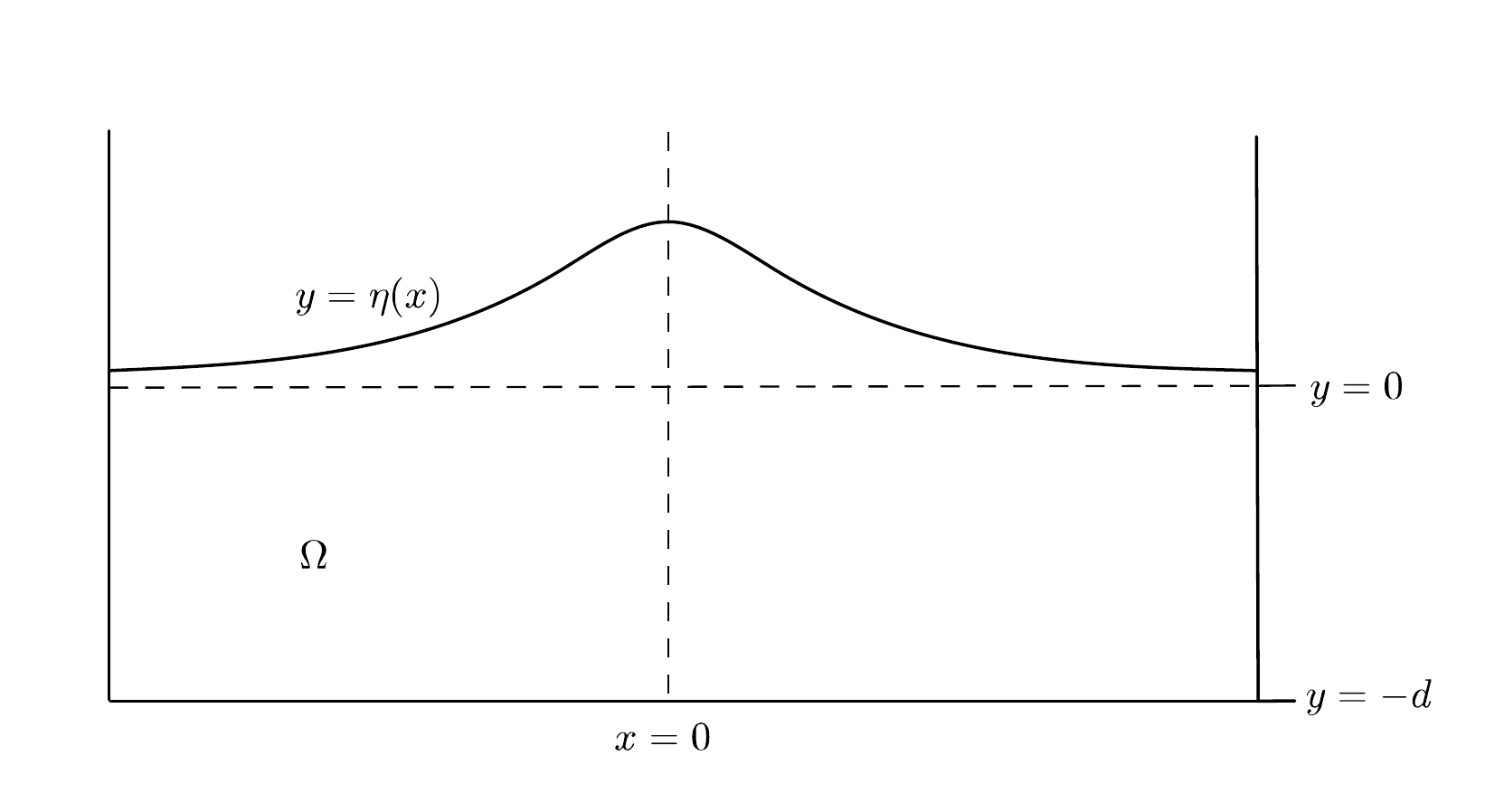}
\caption{Solitary wave in the moving reference frame.}
\label{fig1}
\end{figure}

For water waves, it is physically reasonable to assume that the fluid is incompressible and homogeneous
(with constant density $\rho=1$), which yields the continuity equation
\begin{equation}\label{continuity}
u_x+v_y=0.
\end{equation}
The motion is governed by Euler's equation which, in the moving frame, reduces to
\begin{equation} \label{euler}
\begin{cases}
(u-c)u_x+vu_y = -P_x, & \quad -d\le y \le \eta(x), \\
(u-c)v_x+vv_y = -P_y-g, & \quad -d\le y \le \eta(x),
\end{cases}
\end{equation}
where $P=P(x,y,t)$ is the pressure and $g$ is the constant of gravitational acceleration.
The boundary conditions associated with \eqref{continuity}--\eqref{euler} are
\begin{eqnarray}
&P=P_\textrm{atm} 	& \text{on} \quad y = \eta(x), \label{Patm} \\
&v=(u-c)\eta_x 		&\text{on} \quad y = \eta(x), \label{surface} \\
&v=0  			&\text{on} \quad y = -d. \label{bed}
\end{eqnarray}
Conditions \eqref{surface} and \eqref{bed} respectively express 
the facts that the particles do not cross the water surface and that
the flat bed at the bottom is impermeable. Finally, assuming that the flow is irrotational 
(which is the case in the absence of underlying currents)
yields the additional condition
\begin{equation}\label{irrot}
u_y-v_x=0.
\end{equation}
For a solitary wave, the irrotational water wave equations \eqref{continuity}--\eqref{irrot}
are complemented by the requirements that the water surface be asymptotically flat, and the flow at rest,
as $x\to\pm\infty$. With our choice of coordinates, this reads
\begin{equation}\label{flat}
\eta(x)\to 0 \qquad\text{as}\quad |x|\to\infty,
\end{equation}
\begin{equation}\label{rest}
u(x,y) \to 0 \quad\text{and}\quad v(x,y) \to 0 \qquad\text{as}\quad |x|\to\infty,
\end{equation}
at any depth $y$ in the fluid domain.

It should be noted here that the parameters $c>0$ and $d>0$ cannot be chosen arbitrarily, 
as nontrivial solutions can only exist \cite{amick} if
\[
c>\sqrt{gd}.
\]
Moreover, all solitary waves are symmetric about a single crest, 
with a strictly monotone profile on either side of this crest \cite{craig}, and the convergence
in \eqref{rest} is exponential \cite{amick,mcleod}.

In view of \eqref{continuity}, the problem can be conveniently reformulated in terms of the
stream function $\psi(x,y)$ defined up to a constant by
\[
\psi_x=-v \quad\text{and}\quad \psi_y=u-c.
\]
Clearly, $\psi$ is harmonic in the fluid domain 
\[
\Omega=\{(x,y)\in\R^2 : -d<y<\eta(x), \ x\in\R\}.
\] 
Furthermore, it follows from \eqref{surface} and \eqref{bed} that $\psi$ is
constant on the free surface $y=\eta(x)$ and on the flat bed $y=-d$, respectively.
Hence, imposing the condition 
\[
\psi(0,\eta(0))=0, 
\]
$\psi$ satisfies the free boundary problem
\begin{equation}\label{free}
\begin{cases}
\Delta \psi = 0 & \text{in} \ \Omega,\\
\txt\frac12 |\nabla \psi|^2 + P + gy = C & \text{in} \ \Omega,\\
\psi = 0 & \text{on} \ y=\eta(x),\\
\psi = m & \text{on} \ y=-d,
\end{cases}
\end{equation}
where $C$ and $m$ are physical constants. We recognise Bernoulli's law in the 
second equation. The constant $C$, related to the total energy of the fluid,
can be evaluated in the limit $|x|\to\infty$ 
along the free surface $y=\eta(x)$, which yields
\[
C=\frac{c^2}{2}+P_\textrm{atm}.
\]
The constant $m$, called the relative mass flux of the flow, satisfies
\begin{equation}\label{massflux}
m=\psi(x,-d)=-\int_{-d}^{\eta(x)}\psi_y(x,y)\diff y=\int_{-d}^{\eta(x)}[u(x,s)-c]\diff y,  \quad x\in\R.
\end{equation}
As discussed in \cite{const_esch}, we have $u-c<0$ throughout $\Omega$, and this
inequality extends to $\overline{\Omega}$ except in the case of 
the wave of greatest height, for which the free surface is analytic everywhere outside of 
$x=0$, where the wave profile is continuous but not differentiable, the curve having 
a corner there containing an angle of $2\pi/3$. In this case we have
$u=c$ at the crest, while $u-c<0$ elsewhere in $\overline{\Omega}$. It thus follows from
\eqref{massflux} that $m>0$. In the present work, for technical reasons 
(see Remark~\ref{tech}), we shall restrict our attention to smooth waves only.


\section{Main result}

Our main result concerns the extrema of the dynamic pressure defined as
\begin{equation}\label{dynamic}
p(x,y)=P(x,y)-(P_\textrm{atm} - g y),
\end{equation}
where $P(x,y)$ denotes the total pressure of the fluid in the moving frame.

\begin{theorem}\label{main.thm}
The dynamic pressure in a smooth irrotational solitary wave attains its maximum value at the
crest, is strictly positive throughout the fluid domain, and decays to zero far away from the crest,
provided there is no underlying current.
\end{theorem}

\begin{proof}
To prove the theorem we adapt a method based on the maximum principle for elliptic
partial differential equations (see e.g.~\cite{fraenkel}) that was recently used in \cite{const} 
to study the extrema of the dynamic pressure in a periodic wave train.

We work in the reference frame introduced in Section~\ref{model.sec}, 
which moves at speed $c$ in the direction of propagation of the 
wave. In this reference frame, the wave is stationary and the wave crest is located at $(0,\eta(0))$. 
We first observe that the velocity field enjoys important symmetry properties with respect
to $x=0$. Namely, $u(x,y)$ is symmetric in $x$ and $v(x,y)$ is antisymmetric in $x$, 
see \cite{const_esch}. Hence, in view of the second equation in \eqref{free}, 
the pressure is symmetric with respect to $x=0$ and we only need to study it in the half plane $x\ge0$.

Our proof now proceeds in two steps. We shall first establish the monotonicity properties
of the dynamic pressure function $p(x,y)$ along the boundary of the fluid subdomain
\[
\Omega_+=\{(x,y)\in\R^2 : -d<y<\eta(x), \ x>0\},
\]
as illustrated in Fig.~\ref{fig2}. Namely, $p(x,y)$ decreases strictly along the broken line 
\[
L=\{(0,y) : -d\le y \le \eta(0)\} \cup \{(x,-d): 0\le x<\infty\}
\]
as $(x,y)$ runs from $(0,\eta(0))$ to $(0,-d)$ and from $(0,-d)$ to $(\infty,-d)$. 
Also, $p(x,y)$ decreases strictly as $x$ increases along the free surface $y=\eta(x)$.
Furthermore, we show that $p(x,y)\to 0$ as $x\to\infty$ along each of these curves.
The proof will then be completed by observing that, thanks to the maximum principle, 
the extrema of $p$ cannot be attained in the interior of 
$\overline{\Omega}_+$.

\begin{figure}[th]
\includegraphics[width=130mm]{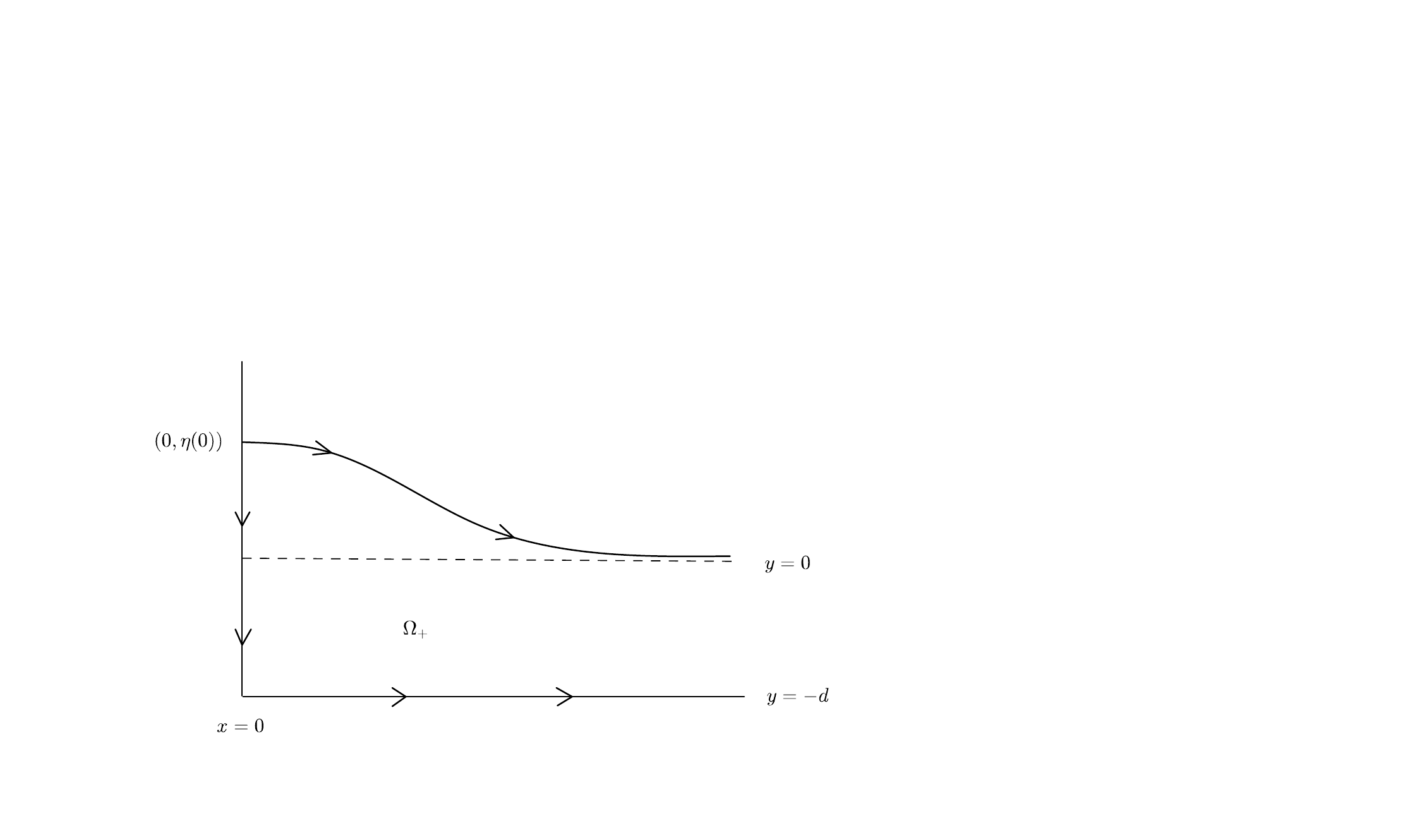}
\caption{Monotonicity along the boundary.}
\label{fig2}
\end{figure}

\smallskip
\noindent
\underline{Step~1:} In view of \eqref{bed} and the antisymmetry of $v$, we have that $v=0$
on $L$, while \eqref{surface} ensures that $v>0$ along $y=\eta(x), \ x>0$. Since $v$ is harmonic
in $\Omega_+$, the maximum principle implies $v>0$ in $\Omega_+$. Now, using the second
equation in \eqref{free}, the formula \eqref{dynamic} reduces to
\begin{equation}\label{dynamic2}
p(x,y)=\frac{c^2}{2}-\frac{v^2+(u-c)^2}{2}.
\end{equation}
Thanks to \eqref{continuity} and \eqref{irrot}, it follows by Hopf's boundary-point 
lemma \cite{fraenkel} that
\begin{align}\label{hopf}
p_x(x,-d)	& = (u(x,-d)-c)v_y(x,-d)<0, \quad x>0 \\
p_y(0,y)	& = -(u(0,y)-c)v_x(0,y)>0, \quad -d<y<\eta(0), \label{hopf2}
\end{align}
which yields the monotonicity property along $L$. On the free surface,
\eqref{Patm} and \eqref{dynamic} simply give $p(x,\eta(x))=g\eta(x)$, 
and the claim follows from the monotonicity of $\eta(x), \ x>0$.
We conclude this first step by observing that, in fact, \eqref{rest} and 
\eqref{dynamic2} imply
\begin{equation}\label{decay}
p(x,y) \to 0 \quad\text{as}\quad |x|\to\infty,
\end{equation}
at any depth $y$ in the fluid domain.

\smallskip
\noindent
\underline{Step~2:} In view of \eqref{dynamic2}, a direct calculation using the fact
that $\psi$ is harmonic shows that 
\begin{equation}\label{super}
p_{xx}+p_{yy}=-\frac{2(p_x^2+p_y^2)}{\psi_x^2+\psi_y^2}\le 0,
\end{equation}
so that $p$ is superharmonic in $\Omega_+$. We deduce from the maximum principle 
that the minimum of $p$ cannot be reached at an interior point of $\Omega_+$ unless
$p$ is constant, which would contradict \eqref{hopf}. It therefore follows from \eqref{decay}
that 
\begin{equation}\label{min}
p>0 \quad\text{in} \quad \Omega_+.
\end{equation}

Furthermore by \eqref{super}, $p(x,y)$ satisfies the quasilinear elliptic equation
\[
p_{xx}+p_{yy}+a(x,y)p_x+b(x,y)p_y=0,
\]
where the variable coefficients $a,b$ are defined throughout $\Omega_+$ by
\[
a=\frac{2p_x}{v^2+(u-c)^2} \quad\text{and}\quad
b=\frac{2p_y}{v^2+(u-c)^2}.
\]
Invoking again the maximum principle, we see that $p$ can only attain its maximum 
on the boundary of $\Omega_+$, which concludes the proof of the theorem.
\end{proof}

\begin{remark}\label{tech}
\rm
We restricted Theorem~\ref{main.thm} to smooth waves because
the coefficients $a(x,y)$ and $b(x,y)$ fail to be bounded 
at the wave crest for the wave of greatest height, so the conclusion about the maximum
of $p(x,y)$ cannot be inferred in a straightforward way. One should consider a finer
analysis in the spirit of \cite{lyons} to handle the wave of greatest height.
Nevertheless, the conclusion \eqref{min} about the minimum still follows from the arguments above 
for the wave of greatest height.

We also point out here that a result in the spirit of Theorem~\ref{main.thm} for solitary waves
with non-zero vorticity would be of great interest. This would involve new challenging
difficulties, for the present approach based on properties of harmonic functions breaks down in this case.
We refer the reader to \cite{var, strauss_wheeler, henry} for some considerations in this direction.
\end{remark}

An important consequence of the above analysis 
for practical applications is the following estimate on the wave
height from pressure measurements at the bed. 

\begin{corollary}\label{height.cor}
The maximum elevation $h$ of a travelling solitary wave
satisfies the lower bound
\begin{equation*}
h>\frac{P(0,-d)-P_\infty}{g},
\end{equation*}
where $P_\infty=P(\infty,-d)$ is the asymptotic value of the pressure at the bed in the fluid at rest.
\end{corollary}

\begin{proof}
On the flat bed $y=-d$, we have by \eqref{dynamic} that 
$p(x,-d)=P(x,-d)-(P_\textrm{atm} + gd)$. Therefore, it follows from 
\eqref{hopf2} and \eqref{decay}  that 
\[
P(0,-d)-P_\infty=p(0,-d)-p(\infty,-d)<p(0,\eta(0))=g\eta(0),
\]
where the last equality comes from \eqref{Patm} and \eqref{dynamic}.
\end{proof}

This result should be put in perspective with \cite{const_recovery}, where the full solitary
wave profile is actually recovered from pressure measurements at the bottom. Note that
Corollary~\ref{height.cor} holds also for the wave of greatest height, as it only relies on the
boundary-point estimate in \eqref{hopf2}.


\end{document}